\documentclass[11pt]{amsart}

\usepackage[dvipsnames]{xcolor}
\usepackage{amssymb,verbatim}
\usepackage{amsmath,amsfonts,enumitem,bm}
\usepackage[mathscr]{euscript} 
\usepackage{amsthm}
\usepackage{url}
\usepackage{graphicx} 
\usepackage{enumitem} 

\usepackage{hyperref} 

\hypersetup{colorlinks} 
\usepackage{bbm} 
\usepackage{bm} 
\usepackage{mathrsfs} 
\usepackage{cancel} 
\usepackage{ytableau} 
\usepackage{mathtools}
\usepackage{faktor} 

\usepackage[colorinlistoftodos]{todonotes}

\RequirePackage{cleveref}
\usepackage{hypcap}
\hypersetup{colorlinks=true, citecolor=darkblue, linkcolor=darkblue}
\definecolor{darkblue}{rgb}{0.0,0,0.7}
\newcommand{\darkblue}{\color{darkblue}}

\definecolor{darkred}{rgb}{0.68,0,0}

\definecolor{darkgreen}{rgb}{0,.38,0}

\newcommand{\defn}[1]{\emph{\darkblue #1}}

\setlist[enumerate]{
	label=\textnormal{({\roman*})},
	ref={\roman*}}

\makeatletter
\def\th@plain{%
	\thm@notefont{}
	\itshape 
}
\def\th@definition{%
	\thm@notefont{}
	\normalfont 
}
\makeatother

\newtheorem{thm}{Theorem}[section]
\newtheorem{lemma}[thm]{Lemma}

\newtheorem*{claim*}{Claim}
\newtheorem{cor}[thm]{Corollary}

\newtheorem{conj}[thm]{Conjecture}
\newtheorem{conv}[thm]{Convention}
\newtheorem{question}[thm]{Question}

\theoremstyle{definition}

\newtheorem{rem}[thm]{Remark}

\numberwithin{figure}{section}
\numberwithin{equation}{section}

\def\bu{\bullet}

\def\zz{\mathbb Z}
\def\nn{\mathbb N}
\def\cc{\mathbb C}

\def\rr{\mathbb R}
\def\qqq{\mathbb Q}

\def\ov{\oa}

\def\Ga{\Gamma}

\def\ga{\gamma}
\def\si{\sigma}

\def\al{\alpha}
\def\be{\beta}

\def\cS{{\textbf{\textit{S}}}}

\def\Fl{{\mathcal F}}

\def\ssu{\subset}

\def\<{\langle}
\def\>{\rangle}

\def\rR{ {\textsc {\rm R} } }

\def\oa{\overrightarrow}

\def\rK{{\mathbf{K}}}
\def\rL{{\mathbf{L}}}

\def\0{{\mathbf 0}}

\def\.{\hskip.06cm}
\def\ts{\hskip.03cm}

\def\pt{\partial}

\def\bz{{\textbf{\textit{z}}}}

\def\bx{{\textbf{\textit{x}}}}

\def\by{{\textbf{\textit{y}}}}

\def\bbe{\textbf{\textit{e}}}

\def\bal{{\boldsymbol{\alpha}}}
\def\bbe{{\boldsymbol{\be}}}
\def\bga{{\boldsymbol{\gamma}}}

\newcommand{\weight}{\mathrm{weight}}

\def\.{\hskip.06cm}
\def\ts{\hskip.03cm}

\def\nin{\noindent}

\newcommand{\textsu}[1]{\textup{\textsf{#1}}}
\usepackage{indentfirst}
\DeclareTextSymbolDefault{\ae}{T1}

\newcommand{\ComCla}[1]{\textup{\textsu{#1}}}

\newcommand{\sharpP}{\ComCla{\#P}}
\newcommand{\SP}{\ComCla{\#P}}

\newcommand{\GapP}{\ComCla{GapP}}

\newcommand{\GapL}{\ComCla{GapL}}

\newcommand{\Sigmap}{\ensuremath{\Sigma^{{\textup{p}}}}}
\newcommand{\Pip}{\ensuremath{\Pi^{{\textup{p}}}}}

\newcommand{\NP}{\ComCla{NP}}
\newcommand{\VP}{\ComCla{VP}}
\newcommand{\VNP}{\ComCla{VNP}}
\newcommand{\BPP}{\ComCla{BPP}}

\newcommand{\coNP}{\ComCla{coNP}}

\renewcommand{\P}{\ComCla{P}}
\newcommand{\CeqP}{\ComCla{C$_=$P}}

\newcommand{\PH}{\ComCla{PH}}

\newcommand{\PSPACE}{\ComCla{PSPACE}}

\newcommand{\EXP}{\ComCla{EXP}}

\newcommand{\AM}{\ComCla{AM}}
\newcommand{\coAM}{\ComCla{coAM}}

\def\SP{\sharpP}

\newcommand{\SZK}{\ComCla{SZK}}
\newcommand{\PZK}{\ComCla{PZK}}
\newcommand{\NE}{\ComCla{NE}}
\newcommand{\coNE}{\ComCla{coNE}}

\def\GRH{\textup{\sc GRH}}

\def\HN{\textup{\sc HN}}
\def\HNP{\textup{\sc HNP}}

\def\rpoly{\textup{\rm poly}}

\def\SV{\textup{\sc SchubertVanishing}}

\def\poly{{\P}}
\def\CEP{{\CeqP}}

\newcommand{\inv}{\operatorname{{\rm inv}}}
\newcommand{\Des}{\operatorname{{\rm Des}}}
\newcommand{\des}{\operatorname{{\rm des}}}

\newcommand{\Sch}{\mathfrak{S}} 

\DeclareMathOperator{\zero}{\mathbf{0}} 

\newcommand{\RC}{{\text {\rm RC} } }

\newcommand{\Sc}{\mathfrak{S}}

\begin{document}

\title[Vanishing of Schubert coefficients]{Vanishing of Schubert coefficients \\ is in $\AM\cap \coAM$ assuming the $\GRH$}

\author[Igor Pak \. \and \. Colleen Robichaux]{Igor Pak$^\star$  \. \and \.  Colleen Robichaux$^\star$}

\makeatletter

\thanks{\thinspace ${\hspace{-.45ex}}^\star$Department of Mathematics,
UCLA, Los Angeles, CA 90095, USA. Email:  \texttt{\{pak,robichaux\}@math.ucla.edu}}

\thanks{\today}

\begin{abstract}
The \emph{Schubert vanishing problem} is a central decision problem in
algebraic combinatorics and Schubert calculus, with applications to
representation theory and enumerative algebraic geometry.  The problem
has been studied for over 50 years in different settings, with
much progress given in the last two decades.

We prove that the Schubert vanishing problem is in~$\AM$ assuming the
\emph{Generalized Riemann Hypothesis} ($\GRH)$.
This complements our earlier result in \cite{PR24o},
that the problem is in~$\coAM$ assuming the $\GRH$.
In particular, this implies that the Schubert vanishing problem is
unlikely to be $\coNP$-hard, as we previously conjectured
in \cite{PR24o}.

The proof is of independent interest as we formalize and expand the
notion of a \emph{lifted formulation} partly inspired by
algebraic computations of Schubert problems,
and \emph{extended formulations} of linear programs.  We use
the result by Mahajan--Vinay \cite{MV97} to show that the determinant
has a lifted formulation of polynomial size.  We combine this
with Purbhoo's algebraic criterion \cite{Purbhoo06} to derive
the result.
\end{abstract}


\maketitle

\section*{Foreword}

Despite appearances, the results of the paper do not require much
of the background to state, see below.  There is, however,
a great deal of combinatorial and algebraic background needed to
understand and appreciate the motivations.  This occupies the rest of
Section~\ref{s:intro}.

The heart of the proof is a combinatorial result that we call
Determinant Lemma~\ref{l:det},
which states that as a polynomial in commuting variables, the determinant
has a lifted formulation of a polynomial size.  Here the lifted formulations
are an algebraic analogue of extended formulations for convex polyhedra,
while the result is a statement in algebraic complexity theory.
We present both the background and a short proof of the Determinant
Lemma in Section~\ref{s:lifted}.

Then, in Section~\ref{s:proof} we present the proof of the main result.
In Appendix~\ref{App:Schub}, we present four different definitions of
Schubert polynomials, elucidating their different properties.
Finally, in Appendix~\ref{s:finrem} we include various quotes
on the history of the problem and its significance.

\smallskip

\section{Schubert vanishing problem}\label{s:intro}

\subsection{Main result} \label{ss:intro-main}
\defn{Schubert polynomials} \ts $\Sch_w\in \nn[x_1,x_2,\ldots]$ \ts indexed by
permutations $w\in S_n\ts$, are celebrated generalizations of Schur polynomials. Schubert polynomials satisfy partial symmetries, but are not symmetric in general.
They were introduced by Lascoux and Sch\"{u}tzenberger~\cite{LS82,LS85}
to represent cohomology classes of Schubert varieties in the complete
flag variety, building on the earlier works by Demazure \cite{Dem74} and
Bernstein--Gelfand--Gelfand \cite{BGG73}.  We refer to \cite{Las95} for
a historical introduction.

Schubert polynomials have been extensively studied from algebraic,
combinatorial, representation theoretic, and computational points of view.
Fundamentally, they represent a combinatorial approach to problems in enumerative
algebraic geometry, answering questions of the form: \emph{How many lines in the space
intersect four given lines in general position?}\footnote{The answer is $2$ in
this case, see e.g.~\cite{KL72}.  Making rigorous sense of the natural generalization of this problem
was the goal of \emph{Hilbert's fifteenth problem} (1900).  Resolving it required
a major effort, resulted in several (equivalent) formal definitions, see e.g.\ \cite{Kle76}. }
See~\cite{Mac91,Man01} for
classic introductory surveys, \cite{Knu16,Knu22} for overviews of recent results,
\cite{AF,KM05} for geometric aspects, and
\cite[$\S$10]{Pak-OPAC} for an overview of computational complexity aspects.

\defn{Schubert coefficients} \ts are defined as structure constants:
$$
\Sch_u \cdot \Sch_v \, = \, \sum_{w \ts \in \ts S_\infty} \. c^w_{u,v} \. \Sch_w\ts,
$$
for $u,v\in S_n$.
It is known that \ts $c^w_{u,v} \in \nn$ \ts for all \ts $u,v,w\in S_\infty$\ts,
as they have a geometric meaning which generalize
the number of intersection points of lines, see e.g.\ \cite{AF,Ful97}.
Schubert polynomials and Schubert coefficients also emerge in
representation theory \cite{BS00,RS95}, category theory \cite{KP04},
matroid theory \cite{AB07}, and the
pole placement problem in linear systems theory \cite{Byr89,EG02}.

Since Schubert polynomials are {Schur
polynomials} for \emph{Grassmannian permutations} (permutations with one descent),
Schubert coefficients generalize the \emph{Littlewood--Richardson coefficients}.
The \defn{Schubert vanishing problem} \ts is a decision problem
$$
\SV \, := \, \big\{c^w_{u,v} =^? 0 \big\}.
$$
This is an extremely well studied problem, both for its own sake, and as a
stepping stone towards understanding the nature of Schubert coefficients.
The main result of this paper is the following:

\smallskip

\begin{thm}[{\rm Main theorem}{}]\label{t:main}
$\SV \. \in \. \AM \cap \coAM$ \. assuming the \ts $\GRH$.
\end{thm}

\smallskip

Here the \ts $\GRH$ \ts stands for the \emph{Generalized Riemann Hypothesis}, that
all nontrivial zeros of $L$-functions \ts $L(s,\chi_k)$ \ts have real part~$\frac12$.
The main theorem comes as a surprising improvement over
the main result in \cite{PR24o}, which states that \ts $\SV \. \in \. \coAM$ \ts
assuming the \ts $\GRH$.
See an extensive discussion of the prior work later in this section.

\smallskip

\subsection{Schubert vanishing} \label{ss:intro-SV}
It turns out, the Schubert vanishing problem can be stated in an
elementary language in the style of classical geometry, without
the use of Schubert polynomials.  Since we will not need to use them,
several formal definitions of Schubert polynomials are included in
Appendix~\ref{App:Schub} in case the reader is curious.

Let $Z=\cc^n$ be a fixed vector space with a basis $\{e_1,\ldots,e_n\}$.
A \emph{complete flag} $F_\bu$ is a sequence of subspaces \ts
$\{\zero\} =F_0\subset F_1 \subset F_2 \subset\ldots \subset F_n = Z$, where $\dim(F_i)=i$
for all \ts $0\le i \le n$.  Let $\Fl_n$ denote the set of complete flags in~$Z$.
A \emph{coordinate flag} $E_\bu \in \Fl_n$ is a complete flag \ts
$\{\zero\}=E_0\subset E_1 \subset E_2 \subset\ldots \subset E_n = Z$, where
$E_i = \cc\<e_1,\ldots,e_i\>$.  A \emph{permutation flag} $E_\bu^w \in \Fl_n$
is a coordinate flag corresponding to the basis $\{e_{w(1)},\ldots,e_{w(n)}\}$.

\bigskip

$\neg\SV$ \ts ($=$ {\sc SchubertNonVanishing})

\smallskip

{\bf Input:} \. $n\times n$ integral matrices $(a_{ij})$,  $(b_{ij})$, $(c_{ij})$

\smallskip

{\bf Decide:} \. $\forall \ts U_\bu, V_\bu \in \Fl_n$ \. s.t.\ \. $U_i \cap V_{n-i} = \{\zero\}$ \. for all \. $1\le i < n$,
\smallskip

\hskip 1.6cm
$\exists \ts W_\bu\in \Fl_n$ \. s.t.\ \. $\dim(W_i \cap E_j)=a_{ij}\ts$, \, $\dim(W_i \cap U_j)\ge b_{ij}$

\smallskip

\hskip 1.6cm  and \. $\dim(W_i \cap V_j)\ge c_{ij}\ts$, \. for all \. $1\le i,j \le n$

\bigskip

\nin
The naming choices for these complete flags is not accidental.  It is easy to see that
for every complete flag $W_\bu$, the \emph{dimension matrix} \ts
$\dim(W_i \cap E_j)=a_{ij}$ \ts coincides with a dimension matrix for
some permutation flag $E_\bu^w$ for some $w\in S_n\ts$, so we have \.
$a_{ij} =  \big| \{w(1),\ldots,w(j)\}  \cap \{1,\ldots,i\}\big|$.
Similarly, taking \. $b_{ij} = \big| \{u(1),\ldots,u(j)\}  \cap \{1,\ldots,i\}\big|$ \.
and  \. $c_{ij} = \big| \{v(1),\ldots,v(j)\}  \cap \{1,\ldots,i\}\big|$ \.
provides a translation between two equivalent formulations of the Schubert vanishing problem.

\smallskip

\subsection{Prior work: general results} \label{ss:intro-prior}
Much of the work on the problem has been a healthy collaboration and occasional
competition of combinatorial and algebraic tools.  Below we give a somewhat
ahistorical overview, leaving the special cases until the end.

As stated in~$\S$\ref{ss:intro-SV} above, the Schubert vanishing problem
is not a priori decidable since there are
uncountably many pairs of complete flags $U_\bu, V_\bu$ to be checked.
In fact, over $\rr$ the existence of $W_\bu$ can depend on $U_\bu,V_\bu$
even if these two complete flags are generic; over large finite fields
a major result by Vakil that this does not happen \cite{Vak06}.

Over $\cc$, the problem simplifies significantly.  We can always assume
that $U_\bu$ is a permutation flag, while $V_\bu$ is a \emph{generic flag},
but making the notion of ``generic'' quantitative is highly nontrivial
and not well-understood in explicit terms.  Heuristically, this phenomenon is
a variation on the \emph{polynomial identity testing} ({\sc PIT}),
where the problem is in $\BPP$ over large finite fields,
while over $\rr$ the problem is believed to be not in $\PH$.\footnote{Over $\rr$,
{\sc PIT} is equivalent to the \emph{existential theory of the reals} \ts
{\sc $(\exists\rr)$}, see e.g.\ \cite{Sch10}.}

By extending combinatorial tools of Lascoux and Sch\"utzenberger, it was shown
in \cite{BB93,BJS93,FS94} that the \emph{Schubert--Kostka numbers}
\ts $K_{w,\al} := [x_1^{\al_1}x_2^{\al_2}\cdots] \ts \Sch_w$ \ts
are nonnegative integers, and moreover that they are in $\SP$ as a counting function.
This immediately shows that $\SV\in \PSPACE$.  The \emph{effective M\"obius inversion}
argument in \cite[Thm~1.4]{PR24} easily implies that computing Schubert coefficients
is in \ts $\GapP=\SP-\SP$, and that \ts $\SV\in \CEP$.  This was
also observed earlier by Morales as a consequence of the \emph{Postnikov--Stanley
formula} \cite[$\S$17]{PS09}, see \cite[Prop.~10.2]{Pak-OPAC} for the explanation.
Until recently, it was believed that \ts $\SV\notin \PH$, and potentially
even $\CEP$-complete, see a discussion in \cite[$\S$2.2]{PR24}.

In a different direction, a direct description of an algebraic system was
given by Billey and Vakil
in \cite[Thm~5.4]{BV08}, which has exactly \ts $c^w_{u,v}$ \ts
solutions for generic values of certain variables.  They also describe
the system of conditions for these variables being generic under the assumption
that the set of solutions is $0$-dimensional \cite[Cor.~5.5]{BV08}.
The authors do not give a complexity analysis for this system;
see \cite[$\S$8.1]{PR24o} for further details and a complexity discussion.

In~\cite{HS17}, Hein and Sottile introduced an
algebraic system  similar in flavor that they called a \emph{lifted square formulation},
giving a practical algorithm for computing Schubert coefficients
\ts $c^w_{u,v}$.  Their system had additional variables compared to the
Billey--Vakil system, and allowed polynomial equations to have smaller
(polynomial) size.  This property was critical in the analysis given
in~\cite{PR24} (see below).  We note that systems in other
\emph{numerical Schubert calculus} papers \cite{HHS16,L+21}
do not have polynomial size and are based on different principles.

Another approach was given by Purbhoo \cite{Purbhoo06}
(see also \cite[$\S$2]{Belkale06}).  He introduced a fundamentally
different algebraic system which gives necessary and sufficient condition for
vanishing of Schubert coefficients.  Just like the Billey--Vakil and
Hein--Sottile system, some variables were required to be generic to ensure
that the set of solutions is $0$-dimensional, a difficult condition to
analyze computationally.  An important feature of this approach is
the dual nature of the system, as it gives an algebraic certificate for
vanishing rather than non-vanishing.

In \cite[Thm~1.4]{PR24o}, we proved that \ts $\SV \in \coAM$ \ts assuming the \ts
$\GRH$.  We modified the Hein--Sottile system to prove the inclusion
\ts $\neg \SV \in \HNP$, the
parametric version of the {\em Hilbert Nullstellensatz}.
Then we used a  recent result in
\cite{A+24}, that \ts $\HNP\in \AM$ \ts assuming the \ts $\GRH$,
which is an extension of Koiran's celebrated result for the
{Hilbert Nullstellensatz} \cite{Koiran96}.

Our Main Theorem~\ref{t:main} that \ts $\SV \in \AM$, is a complementary result
proved by using a superficially similar inclusion \ts $\SV \in \HNP$.
The proof is based on Purbhoo's algebraic system.  Curiously, 
we also used Purbhoo's algebraic system to show that \ts
$\neg\SV \in \NP_\cc \cap \poly_\rr\ts$, a complexity result
in the \emph{Blum--Shub--Smale model of computation} over general fields
 \cite[Appendix~B]{PR24o}.

\smallskip

\subsection{Prior work: special cases} \label{ss:intro-special}
There is a large number of sufficient conditions for vanishing of
Schubert coefficients scattered across the literature.  These were
made both in an attempt to better understand the problem from a
combinatorial point of view, and as a partial result towards its
eventual resolution. Immediate from the combinatorial and geometric interpretations of the Schubert coefficients, we have the \emph{dimension condition} that if $\inv(u)+\inv(v)\neq \inv(w)$ then $c_{u,v}^w=0$, where $\inv$ denotes the number of inversions of the permutation.

Further, we have the following conditions,
which can be verified in polynomial time:

\smallskip

$\circ$ \ts  the \emph{number of descents condition} of Lascoux and Sch\"utzenberger \cite{LS82},

$\circ$ \ts \emph{strong Bruhat order condition},
see e.g.\ \cite[$\S$5.1]{StDY22} combined with \cite[Prop.~2.1.11]{Man01},


$\circ$ \ts Knutson's \emph{descent cycling condition} \cite{Knutson01} (see also
\cite[Cor.~4.15]{PW24}),

%
$\circ$ \ts \emph{permutation array condition} by  Billey and Vakil \cite[Thm~5.1]{BV08} (see also \cite[Prop.~9.7]{AB07}),

$\circ$ \ts St.~Dizier and Yong's condition on certain filling of Rothe diagrams \cite[Thm~A]{StDY22}, and


$\circ$ \ts Hardt and Wallach's condition on empty rows in Rothe diagrams \cite[Cor.~5.12]{HW24}.

\smallskip
\nin
For Grassmannian permutations, Schubert coefficients are the \emph{Littlewood--Richardson}
(LR) \emph{coefficients}, see e.g.\ \cite{Mac91,Man01}.
In this special case the vanishing problem is in~$\poly$ as a corollary
of the Knutson--Tao \emph{saturation theorem} \cite{DeLM06,MNS12}.
This is one of several important special cases where Schubert coefficients  \ts $c^w_{u,v}$ \ts
have a known combinatorial interpretation.  In such cases, the combinatorial interpretation
can be interpreted as $\NP$ (sufficient) conditions for non-vanishing.  Notable examples include:

\smallskip

$\circ$ \ts Purbhoo's \emph{root game conditions} \cite{PurbhooThesis,Purbhoo06}, and

$\circ$ \ts Knutson and Zinn-Justin's several \emph{tiling conditions} \cite{KZ17,KZ23}.

\smallskip

\nin
We refer to \cite[$\S$5]{StDY22} and \cite[$\S$1.6]{PR24o} for technical details,
comparisons, and further background on all these conditions.

\smallskip

\subsection{Implications} \label{ss:intro-imply}
Let us emphasize several implications of the main result.

\subsubsection{New type of problem} \label{sss:intro-imply-example}
From the computational complexity point of view, having a new natural problem
in \ts $\AM\cap \coAM$ \ts is quite curious since this problem is apparently
different from known problems in this class.
Indeed, other problems in \ts $\AM\cap \coAM$ \ts
include various ``equivalence problems'':
\emph{graph isomorphism} \cite{GMW91} (see also \cite{BHZ87}),
\emph{code equivalence} \cite{PR97} (see also \cite{BW24}),
\emph{ring isomorphism} \cite{KS06},
\emph{permutation group isomorphism} \cite{BCGQ11}, and
\emph{tensor isomorphism} \cite{GQ23}.

In fact, the problems listed above are in lower complexity classes.
For example, famously, {\sc Graph Isomorphism} is in \ts
$\NP\cap\coAM$, see e.g.\ \cite{KST93}, and is in the perfect/statistical zero knowledge
classes \ts $\PZK \subseteq \SZK \subseteq \AM\cap \coAM$ \cite{Vad99}.
We refer to \cite{BBM11} for a rare example of a
problem that is in $\AM\cap \coAM$, but not necessarily in~$\SZK$.

\subsubsection{Positive rule via derandomization} \label{sss:intro-imply-derandom}
A major open problem in algebraic combinatorics is whether Schubert coefficients
have a \emph{combinatorial interpretation} \cite[Problem~11]{Sta00}, see also
$\S$\ref{ss:finrem-hist} and~$\S$\ref{ss:finrem-nonSP}.  In the language of
computational complexity this is asking whether this counting problem is in
$\SP$, see a detailed discussion in \cite[$\S$10]{Pak-OPAC} (cf.\ also~\cite{Ass23}).
This would imply that $\SV$ is in $\coNP$.  In the combinatorial language,
this is saying that positivity of Schubert coefficients problem \ts
$\big\{c^w_{u,v}>^?0\big\}$ \ts has a \emph{positive rule} \cite{PR24b}.

The special cases mentioned in $\S$\ref{ss:intro-special}
suggest that both Schubert vanishing and Schubert non-vanishing might
have a positive rule, i.e., that \ts $\SV \in \NP\cap\coNP$.  Until recently this
conclusion would seem fantastical and out of reach. Now, it was pointed out
in \cite{PR24o,PR24b}, that a derandomization result by
Miltersen--Vinodchandran \cite{MV05} (extending \cite{KvM02}),
implies that \ts $\AM = \NP$ \ts assuming some languages in
$\NE \cap \coNE$ require nondeterministic exponential size circuits.

For our purposes, a weaker derandomization assumption would also suffice:
it was shown by Gutfreund, Shaltiel and Ta-Shma \cite{GST03}, that if $\EXP$ requires
exponential time even for $\AM$ protocols (we call this {\sc GST} assumption),
then \ts $\AM\cap \coAM =\NP\cap\coNP$.
In a combinatorial language, this {\sc GST} assumption combined with the $\GRH$ imply
that there exists a positive rule for both vanishing and positivity of Schubert
coefficients.  While the latter is quite natural from a combinatorial interpretation
of Schubert coefficients point of view, the former is quite surprising, see below.

\subsubsection{Computational hardness of Schubert vanishing} \label{sss:intro-imply-hard}
It has been known for a while that the vanishing of Schubert coefficients is
computationally hard, see e.g.\ an extensive discussion in \cite[$\S$5.2]{BV08}
and~$\S$\ref{ss:finrem-complexity}.  Here are two versions of the problem available
in the literature.

\smallskip

\begin{question}
[{\rm Adve, Robichaux and Yong \cite[Question~4.3]{ARY19}}{}]  \label{conj:SV-NP}
Is \ts $\SV$ \ts $\NP$-hard?
\end{question}


\begin{conj}[{\rm Pak and Robichaux \cite[Conj.~1.6]{PR24o}}{}]  \label{conj:SV-CoNP}
$\SV$ \ts is \ts $\coNP$-hard.
\end{conj}

\smallskip

The following result resolved both the question and the conjecture under standard assumptions:

\smallskip

\begin{cor}  \label{c:SV-not-hard}
$\SV$ \ts is not \ts $\NP$-hard, assuming the \ts $\GRH$ \ts and \ts $\PH\ne \Sigmap_2$.
Similarly, \ts
$\SV$ \ts is not \ts $\coNP$-hard, assuming the \ts $\GRH$ \ts and \ts $\PH\ne \Pip_2$.
\end{cor}

\smallskip

\nin
The corollary follows immediately from the Main Theorem~\ref{t:main} and a result
of Boppana, H{\aa}stad and Zachos \cite[Thm~2.3]{BHZ87}.  The corollary implies
that the vanishing of Schubert coefficients is quite different from the vanishing
of \emph{Kronecker coefficients}, which is known to be $\coNP$-hard
even for partitions given in unary \cite{IMW17}.  See also \cite[$\S$5.2]{Pan23}
for further details and references.

\smallskip

\subsection{Notation} \label{ss:intro-notation}
We use \ts $\nn=\{0,1,2,\ldots\}
$ \ts and \ts $[n]=\{1,\ldots,n\}$.
We use \ts $e_1,\ldots,e_n$ \ts to denote the standard
basis in~$\cc^n$, and $\zero$ to denote zero vector.
We use bold symbols such as \ts $\bx$, $\by$, $\bal$, $\bbe$ \ts
to denote sets and vectors of variables, and bars such as \ts $\ov x$ \ts and \ts $\ov y$,
to denote complex vectors.  We also use \ts $\ov f$ \ts to denote a sequence
of polynomials \ts $(f_1,\ldots,f_m)$.

In computational complexity, we use only standard notation and
complexity classes.  We refer to \cite{AB09,Gol08,Pap94} for the definitions
and extensive background, and to \cite{Aar16} for the extensive introduction.

\medskip



\section{Lifted and compact formulations}\label{s:lifted}

The following section is completely independent of the rest of the paper
and discusses Hilbert's Nullstellensatz and the technology of lifted formulations.

\smallskip

\subsection{Hilbert's Nullstellensatz}\label{ss:lifted-HN}
Let $\rK=\cc[x_1,\dots,x_k]$ for some $s>0$.  We use $\bx = (x_1,\ldots,x_k)$.
Consider a system
\begin{equation}\label{eq:HN-system}
f_1(\bx) \. = \. \ldots \. =  \. f_m(\bx) \. = \. 0  \quad \text{where} \quad f_i\in \rK\ts,
\end{equation}
and denote by \ts $\cS\big(\ov f\big)\subseteq \cc^k$ \ts the set of solutions, where \ts $\ov f = (f_1,\ldots,f_m)$.

\emph{Hilbert's weak Nullstellensatz} \ts
is a fundamental result in algebra, which states that a polynomial system
has no solutions over $\cc$ \. \underline{if and only if} \.
there exist \ts $(g_1,\ldots,g_m)\in \rK^m$, such that
$$
\sum_{i=1}^m \. f_i \. g_i \, = \, 1\ts.
$$
Now let \ts $f_1,\ldots,f_m\in \zz[x_1,\dots,x_k]$.
The decision problem $\HN$ (\defn{Hilbert's Nullstellensatz}),
asks if the polynomial system \eqref{eq:HN-system}
has a solution over \ts $\mathbb{C}$.\footnote{By the Nullstellensatz,
this is equivalent to asking if there is a solution over $\overline{\qqq}$.
}
Here and everywhere below, the \emph{size} \ts %
of the polynomial system~\eqref{eq:HN-system} is defined as
$$
\phi\big(\ov f\big) \. := \. \sum_{i=1}^{m} \. \deg(f_i) \. + \. \sum_{i=1}^{m} \. s(f_i),
$$
where $s(g)$ denotes the sum of bit-lengths of coefficients in the polynomial~$g$.
Famously, Koiran showed that  \ts $\HN$ \ts is in the
polynomial hierarchy:

\smallskip

\begin{thm}[{\cite[Thm~2]{Koiran96}}{}]\label{t:main-HN}
    \. $\HN$ \ts is in \ts $\AM$ \ts assuming the \ts $\GRH$.
\end{thm}

\smallskip

For the proof, Koiran's needs the existence of primes in certain
intervals and with modular conditions, thus the $\GRH$ assumption.
We refer to \cite{A+24} for detailed overview of the problem and
references to the earlier work.

\smallskip

\subsection{Lifted formulations}\label{ss:lifted-def}
Let $\rL=\cc[x_1,\dots,x_k,y_1,\ldots,y_\ell]$ for some $k,\ell>0$.  We say that a system
\begin{equation}\label{eq:lifted}
g_1(\bx,\by) \. = \. \ldots \. =  \. g_\ell(\bx,\by) \. = \. 0  \quad \text{where} \quad g_i\in \rL
\end{equation}
is a \defn{lifted formulation} of the system \eqref{eq:HN-system}, if
\begin{equation}\label{eq:lifted-cond}
\forall \ts \bx \in \cS\big(\ov f\big) \, \exists \ts \by \in \cc^\ell \ : \ (\bx,\by) \in \cS\big(\ov g\big).
\end{equation}
In other words, a natural projection \ts $\iota: \cc^{k+\ell}\to \cc^k$ \ts maps solutions of
\eqref{eq:HN-system} into solutions of \eqref{eq:lifted}: \ts
$\iota\big(\cS(\ov g)\big) = \cS\big(\ov f\big)$.  Clearly, $\HN$ for the system
\eqref{eq:HN-system} is equivalent to $\HN$ for  the system \eqref{eq:lifted}.

It may seem unintuitive that the size of a system of polynomials can become
smaller for a lifted formulation, but this is not uncommon in symmetric situations.
For example, polynomial \ts $x^{2^r}=z$ \ts has exponential size (in~$r$), while its
lifted formulation
$$
y_1 = x^2\,, \, y_2 = y_1^2\,, \, \ldots \,, \, y_r = y_{r-1}^2\,, \, y_r=z
$$
has linear size.

The idea of lifted formulations is completely standard in numerical
analysis and especially numerical algebraic geometry, see e.g.\ \cite{HS17,L+21}
for some recent work in the context of Schubert polynomials.  Our own motivation
comes from the literature on extended formulations of polyhedra, see below.

Lifted formulation \eqref{eq:lifted} is called \defn{compact} \ts if \.
$\phi(\ov g) \ts = \ts \rpoly(k)$.  A major obstacle in \cite{PR24o}, see Remark~A.2,
was that the determinant polynomial equation has exponential size.  The following
lemma shows that the determinant has a compact lifted formulation.

\smallskip

\begin{lemma}[{\rm Determinant lemma}{}]\label{l:det}
Let \ts $X=(x_{ij})$ \ts be an $n \times n$ matrix of variables.  Then
the {\em \defn{determinant equation}}
\begin{equation}\label{eq:det}\tag{Det}
\det X \. = \. z
\end{equation}
has a lifted formulation of size $O(n^3)$.
\end{lemma}

\smallskip

This lemma is key for our proof of Theorem~\ref{t:main} as we use it
essentially as a black box.  We expect it to have further geometric
and representation theoretic applications in the future.

\smallskip

\subsection{Prior work: extended formulations} \label{ss:lifted-extended}
We modeled the notions above after extended formulations,
the celebrated area of combinatorial optimization.

For a convex polyhedron $P \ssu \rr^k$ defined by inequalities, an
\emph{extended formulation} is a polyhedron $Q \ssu \rr^\ell$
which projects onto~$P$.  When $Q$ is defined by $\rpoly(k)$ inequalities,
such extended formulation is called \emph{compact}.  In such cases,
one efficiently solve \emph{linear programming} problems on~$P$,
by solving them on~$Q$ and projecting solutions onto~$P$,
even if $P$ has exponentially many facets.

A prototypical example is the \emph{permutohedron} $P_n \ssu \rr^n$
defined as a convex hull of points $\big(\si(1),\ldots,\si(n)\big)$,
where $\si \in S_n\ts$.  This polytope is $(n-1)$-dimensional,
has $n!$ vertices and $(2^n-2)$ defining inequalities.  The permutohedron
has a compact extended formulation as a projection of the
\emph{Birkhoff polytope} \ts $B_n\ssu \rr^{n^2}$, see e.g.\ \cite{Pas14}.
This is a $(n-1)^2$-dimensional polytope of bistochastic matrices,
with $n!$ vertices given by $0/1$ matrices, and with $n^2$ defining inequalities.

The idea of using extended formulations was formalized by Yannakakis
in a remarkable paper \cite{Yan91}, where he showed that some combinatorial
polytopes have compact extended formulations.  More importantly,
Yannakakis shows that matching and {\sc TSP} polytopes, \emph{do not}
have compact extended formulations, under certain ``symmetry'' constraints.
One can view the latter result as an obstacle to solving {\sc TSP},
a benchmark $\NP$-hard problem, using a polynomial size linear program.

These results inspired a long series of papers.  Notably, the notion of
extended formulations was generalized to \emph{positive semidefinite lifts}
in \cite{GPT13}.  The symmetry constraints were eventually removed in major advances
\cite{F+15,Rot17}.  We refer to \cite{CCZ13,CCZ14,KWY11} for the introduction to the area,
background in combinatorial optimization and further references.

Our proof of the Determinant Lemma~\ref{l:det} is partially motivated by a
beautiful construction of an extended formulation of size $O(n\log n)$ given
by Goemans \cite{Goe15}, which is optimal and improves the $O(n^2)$ size of
the Birkhoff polytope construction.  His construction employs the structure
and efficiency of the \emph{AKS sorting} \cite{AKS83}, to simulate
the sorting with linear inequalities.  In effect, Goemans's construction is
completely oblivious to the exact details the AKS sorting, which in turn
uses explicit construction of expanders also as a black box.

\smallskip

\subsection{Prior work: complexity of the determinant} \label{ss:lifted-det-complexity}
Computing the permanent vs.\ determinant is a classical problem going back to
Valiant \cite{Val79b}, leading to $\VP =^? \VNP$ problem.  In turn, this problem
led to foundations of the \emph{Geometric Complexity Theory} (GCT), see e.g.\ \cite{Aar16,BI18}.
Of course, much of the effort is on lower bounds, a subject tangential to this paper.

There are several different ways one can restrict the computational model, e.g.\
\cite{LR17} gives an exponential lower bound for the permanent under symmetry
constraints.  Similarly, the
\emph{Algebraic Complexity Theory} (ACT) restricts algebraic computations to
straight line programs, and is the closest to our need.  We refer to \cite{BCS97} for
the careful treatment of ACT and the background, and to \cite{CKL24} for a recent
treatment of the {\sc Determinant} in the context of (more general) algebraic branching
problems.

We note that every straight-line
program can be realized as a lifted formulation, but the smallest size lifted
formulation can in principle be much smaller.  This is not unusual; see \cite{IL17}
by Ikenmeyer and Landsberg, which compares complexity of computing the
determinant and the permanent for different computational models.  We note
that commutativity of variables is crucial in this setting, since computing
determinant over noncommutative rings requires exponential time \cite{Nis91}.

In the standard model, the differences between permanent vs.\ determinant
become stark.  Famously, the permanent of integer matrices is $\SP$-complete \cite{Val79a},
while the determinant is $\GapL$-complete \cite{Toda91,Vin91}.  Another natural complete
problem in $\GapL$ is the \ts {\sc DirectedPathDifference}, which inputs an
acyclic digraph with a source~$s$, two sinks $t_+\ts,t_-\ts$, and outputs the difference
in the number of paths $s\to t_+$ and $s\to t_-\ts$.  In a beautiful paper \cite{MV97},
Mahajan and Vinay gave a parsimonious reduction of  {\sc DirectedPathDifference} from the
{\sc Determinant}, which we use in our proof of Lemma~\ref{l:det}.

\smallskip

\subsection{Proof of the Determinant Lemma~\ref{l:det}} \label{ss:lifted-det-proof}
Recall the construction in \cite{MV97}; see also \cite[$\S$3]{IL17} for a
concise presentation and examples.  The authors construct an explicit
directed acyclic graph $\Ga_n$ with the following properties:

\smallskip

$\circ$ \ts vertices of $\Ga_n$ have \emph{layers} \ts $\{0,\ldots,n\}$,
and directed edges connect layers $\ell$ to~$(\ell+1)$,

\smallskip

$\circ$ \ts $\Ga_n$ has $O(n^3)$ vertices and $O(n^4)$ edges,

\smallskip

$\circ$ \ts all edges in $\Ga_n$ have weights $x_{ij}\ts$,

\smallskip

$\circ$ \ts $\Ga_n$ has a unique source~$s$ at layer $0$ and two sinks $t_+\ts,t_-$ at layer~$n$,

\smallskip

$\circ$ \ts the sum of weighted paths $s\to t_+$ minus the sum of weighted paths $s\to t_-$ is $\det(X)$.

\smallskip

\nin
Here the weight of a path is a product of weights of its edges.  In \cite{MV97}, the
authors use
this construction to show that $\det(X)$ can be computed in polynomial time using a
straight line program (in contrast with the \emph{Gaussian algorithm} which requires
a circuit).  We use the same graph construction to construct a compact lifted
formulation as follows.

For a vertex $v$ in $\Ga_n$, denote by $y_v$ the corresponding variables.  Start
the lifted formulation with an equation \ts $y_s=1$.  For every vertex $v\ne s$, add
an equation
$$
y_v \, = \, \sum_{(w,v)\in \Ga_n} \. y_w \ts \cdot \ts \weight(w,v)\.,
$$
where the summation is over all directed edges $(w,v)$ in the graph~$\Ga_n\ts$.
Finally, add an equation
$$
z \, = \, y_{t_+} \. - \. y_{t_-}\..
$$
By construction, each $y_v$ counts the sum of weighted paths $s\to v$.  Thus,
we obtain a lifted formulation for \eqref{eq:det} of size $O(n^3)$.
\qed

\smallskip

{
\begin{rem}\label{r:lifted-obli}
Note that just like in the Goemans's construction in \cite{Goe15}, the specifics
of $\Ga_n$ are irrelevant, only the explicit nature and polynomial size are important.
As we mentioned above, any straight-line computation of a polynomial $f(\bx)$ can be
simulated by a lifted formulation, but not vice versa.  It would be interesting to
see if the $O(n^3)$ bound in the lemma can be improved.  Given that we have $(n^2+1)$
variables, can one find a better lower bound for the smallest lifted formulation
of \eqref{eq:det}?
\end{rem}
}


\medskip

\section{Proof of the main theorem} \label{s:proof}

\medskip
\subsection{Parametric Hilbert's Nullstellensatz}\label{ss:intro-HNP}
For the proof of the Main Theorem~\ref{t:main}, we need the following
strengthening of Theorem~\ref{t:main-HN} to finite algebraic extensions.
Let
$$
f_1,\ldots,f_m\.\in\. \mathbb{Z}(y_1,\ldots,y_k)[x_1,\dots,x_s]\ts.
$$
The decision problem $\HNP$ (\defn{Parametric Hilbert's Nullstellensatz})
asks if the polynomial system \eqref{eq:HN-system} has a solution over
\ts $\overline{\mathbb{C}(y_1,\ldots,y_k)}$.
In a remarkable recent work, Ait El Manssour, Balaji, Nosan, Shirmohammadi, and Worrell
extended Theorem~\ref{t:main-HN} to~$\HNP:$

\smallskip

\begin{thm}[{\cite[Thm~1]{A+24}}{}]\label{t:main-HNP}
    \. $\HNP$ \ts is in \ts $\AM$ \ts assuming the \ts $\GRH$.
\end{thm}

\smallskip

We prove Theorem~\ref{t:main} as an application of Theorem~\ref{t:main-HNP}.
More precisely, recall that in \cite{PR24o} the authors proved that
\. $\neg\SV$ \. reduces to \. $\HNP$.  This and Theorem~\ref{t:main-HNP}
immediately imply that \ts $\SV\in \coAM$ \ts assuming the \ts $\GRH$.
In this paper we prove the following
counterpart:

\smallskip

\begin{lemma}[{\rm Main lemma}{}] \label{l:SV-HNP}
\. $\SV$ \. reduces to \. $\HNP$.
\end{lemma}

The lemma, combined with Theorem~\ref{t:main-HNP}, immediately implies Main Theorem~\ref{t:main}.

\smallskip

\subsection{Purbhoo's criterion}
\label{ss:proof-purbhoo}
Let \ts ${\sf G}={\sf GL}_n(\cc)$ \ts  be the \emph{general linear group}.
This is a matrix group lying in an ambient vector space $V \simeq \cc^{n^2}$.
 Let ${\sf B}$ \ts denote the \emph{Borel subgroup}, i.e.\ the group
 of upper triangular matrices.
Let ${\sf N}$ \ts denote the subgroup of \emph{unipotent matrices},
i.e.\ the group of upper triangular matrices with $1$'s on the diagonal. We have:
$$
{\sf N}\subset {\sf B}\subset{\sf G} \subset V.
$$
Let ${\mathfrak n}$ denote the Lie algebra of ${\sf N}$, i.e.\
the set of strictly upper triangular matrices (with $0$'s on the diagonal).
We think of ${\mathfrak n}$ as a subspace of~$V$.

For a permutation \ts $w\in S_n\ts$, define \ts $R_w:={\mathfrak n}\cap (w\cdot {\sf B}_{-})$,
where \ts ${\sf B}_{-}={\sf B}^T$ \ts is the subgroup of lower triangular matrices.
One can think of \ts $R_w=(r_{ij})$ \ts as strictly upper triangular
matrices with \ts $r_{ij}=0$ \ts for all $i<j$ \ts and \ts $w(i) < w(j)$.
For example, we have $R_e=\{\zero\}$ for an identity
permutation \ts $e=(1,2,...,n)$, and \ts $R_{w_\circ} = {\mathfrak n}$ \ts for
the long permutation \ts $w_\circ=(n,n-1,\ldots, 1)$.
\smallskip

\begin{lemma}[{\rm \defn{Purbhoo's criterion} \cite[Corollary~2.6]{Purbhoo06}}{}]\label{lem:Pur}
For generic \ts $\rho,\omega,\tau\in{\sf N}$, we have:
    \[c_{u,v}^w\. \neq \. 0 \quad \Longleftrightarrow \quad
    \rho R_u\rho^{-1}+\omega R_v\omega^{-1}+\tau R_{w_\circ w}\tau^{-1} \, = \, {\mathfrak n}.
    \]
Here the sum is the usual sum of vector subspaces of~$V$.
\end{lemma}

\smallskip

\subsection{Proof of Main Lemma~\ref{l:SV-HNP}}

\smallskip

Considering the converse of Lemma~\ref{lem:Pur}, we consider the equation
\begin{equation}\label{eq:purRHS}
   c_{u,v}^w\. = \. 0 \quad \Longleftrightarrow \quad
    \rho R_u\rho^{-1}+\omega R_v\omega^{-1}+\tau R_{w_\circ w}\tau^{-1} \, \subsetneq \, {\mathfrak n}.
\end{equation}

By the dimension condition we assume $\inv(u)+\inv(v)=\inv(w)$.
Then define
\begin{align*}
    S_u&:=\{x_{ij}{\sf e}_{ij} \, : \, i<j, u(i)>u(j)\},\\
    S_v&:=\{y_{ij}{\sf e}_{ij} \, : \, i<j, v(i)>v(j)\}, \text{ and}\\
    S_{w_\circ w}&:=\{z_{ij}{\sf e}_{ij} \, : \, i<j, (w_\circ w)(i)>(w_\circ w)(j)\}
\end{align*}
Here ${\sf e}_{ij}$ is the $n\times n$ matrix with a $1$ in position $(i,j)$ and $0$'s elsewhere. Let $\bx,\by,\bz$ denote sets of those variables $x_{ij},y_{ij},z_{ij}$ appearing therein. Then these sets form bases of $R_u$, $R_v$, and $R_{w_\circ w}$, respectively. Note that for $\pi\in S_n$,  $\dim(R_{\pi})=\inv(\pi)$. Thus since $\inv(u)+\inv(v)=\inv(w)$, $\dim(R_{u})+\dim(R_{v})+\dim(R_{w_{\circ}w})=\binom{n}{2}$.

To construct generic $\rho\in{\sf N}$, let $\rho=(\rho_{ij})$, where
\[\rho_{ij}=
\begin{cases}
    \alpha_{ij} & \text{ if } \ i<j,\\
    1 & \text{ if }\  i=j,\\
    0 & \text{ otherwise. }
\end{cases}
\]
Here these $\al_{ij}$ are formal parameters.
Similarly build generic $\omega,\tau\in{\sf N}$ in terms of parameters $\beta_{ij}, \gamma_{ij}$, respectively. Then define the sets of these parameters \ts $\bal=\{\al_{ij}\}$, \ts $\bbe=\{\be_{ij}\}$, and \ts $\bga=\{\ga_{ij}\}$, where the indices range over \ts $1\leq i<j\leq n$.

Again, we form $\widetilde{\rho}=(\widetilde{\rho}_{ij})$, where
\[\widetilde{\rho}_{ij}=
\begin{cases}
    a_{ij} & \text{ if } \ i<j,\\
    1 & \text{ if } \ i=j,\\
    0 & \text{ otherwise. }
\end{cases}
\]
Note that here we treat $\al_{ij}$ as variables. Similarly build generic $\widetilde{\omega},\widetilde{\tau}$ in terms of variables $b_{ij}, c_{ij}$, respectively.
Then define the sets of variables \ts ${\textbf{\textit{a}}}=\{a_{ij}\}$, \ts ${\textbf{\textit{b}}}=\{b_{ij}\}$, and  \ts ${\textbf{\textit{c}}}=\{c_{ij}\}$, where the indices range over \ts $1\leq i<j\leq n$.

Then we obtain bases for $\rho R_u\widetilde{\rho}$, $\omega R_v\widetilde{\omega}$, and $\tau R_{w_\circ w}\widetilde{\tau}$, respectively:
\begin{align*}
    T_u&:=\rho  S_u  \widetilde{\rho},\\
    T_v&:=\omega  S_v  \widetilde{\omega}, \text{ and}\\
    T_{w_\circ w}&:=\tau  S_{w_\circ w}  \widetilde{\tau}.
\end{align*}
Note $\dim(R_{u})+\dim(R_{v})+\dim(R_{w_{\circ}w})=\#T_u+ \#T_v+ \# T_{w_\circ w}=\binom{n}{2}=\dim({\mathfrak n})$.

By ignoring all entries weakly below the main diagonal in each matrix, we view each element in $T:=T_u\cup T_v\cup T_{w_\circ w}$ as a ${\binom{n}{2}}$-vector.
Let $M$ be the $\binom{n}{2}\times \binom{n}{2}$ matrix formed by the vectors in $T$.
Then the right-hand side of Equation~\eqref{eq:purRHS} holds if and only if $\det(M)=0$, when we set $\widetilde{\rho}=\rho^{-1}$, $\widetilde{\omega}=\omega^{-1}$, and $ \widetilde{\tau}=\tau^{-1}$.

Let $\mathcal {S}(u,v,w_\circ w)$ be the system formed by the constraints:
$$
\left\{ \ \aligned
& \rho\cdot \widetilde{\rho}\. = \. {\sf Id}_n\., \\
& \omega\cdot \widetilde{\omega} \. = \. {\sf Id}_n\., \\
& \tau\cdot \widetilde{\tau} \. = \. {\sf Id}_n\., \\
&\det(M) \. = \.  0\ts,
\endaligned
\right.
$$
where the last constraint is replaced by its lifted formulation using the Determinant Lemma~\ref{l:det}.
Here $\mathcal {S}(u,v,w_\circ w)$ uses variables ${\textbf{\textit{a}}}\cup {\textbf{\textit{b}}}\cup {\textbf{\textit{c}}}\cup \bx\cup\by\cup\bz$ and parameters $\bal\cup \bbe\cup \bga$.

Note that matrix entries in $M$ have size $O(n^2)$, as they are monomials of degree at most $2$ whose nonzero coefficients are in $\bal\cup\bbe\cup\bga$.
By the Determinant Lemma~\ref{l:det}, the last equation and thus the whole system \ts
$\mathcal {S}(u,v,w_\circ w)$ \ts has size $O(n^{12})$.

Now, the right-hand side of Equation~\eqref{eq:purRHS} holds if and only if $\mathcal {S}(u,v,w_\circ w)$ is satisfiable over $\mathbb{C}(\bal,\bbe,\bga)$.
Since $\bal\cup \bbe\cup \bga$ are algebraically independent, $\mathcal {S}(u,v,w_\circ w)$ has a solution over $\mathbb{C}(\bal,\bbe,\bga)$ if and only if $\mathcal {S}(u,v,w_\circ w)$  has a solution over $\mathbb{C}$ for
a generic choice of evaluations $\ov \al$, $\ov \be$, $\ov \ga$ of $\bal$, $\bbe$, $\bga$. Thus by Lemma~\ref{lem:Pur}, the result follows. \qed

\medskip

\subsection{Further applications}
As noted in $\S$\ref{ss:intro-special} (see also~$\S$\ref{App:Schub-geom}), the Schubert structure constants $c_{u,v}^w$ are also the structure constants arising from multiplying Schubert classes in the cohomology ring of the complete flag variety \cite{LS82}. More generally, we may consider structure constants arising from multiplying Schubert classes in $H^*({\sf G/\sf B})$, for a complex reductive Lie group ${\sf G}\in \{{\sf GL}_n,{\sf SO}_{2n+1},{\sf Sp}_{2n},{\sf SO}_{2n}\}$. Here ${\sf B}\subset {\sf G}$ the Borel subgroup, the subgroup of upper triangular matrices in ${\sf G}$.

These generalized flag varieties ${\sf G/\sf B}$ may be referred to the type $A, B, C, D$ \ts flag varieties, respectively. Let $c_{u,v}^w(Y)$ denote the corresponding type $Y$  structure constants,
where $Y\in\{A, B, C, D\}$.\footnote{For non-classical types \ts $E_6$, \ts $E_7$,
\ts $E_8$, \ts $F_4$ \ts and \ts $G_2\ts$,  there is only a finite number of
Schubert coefficients, so the problem is uninteresting from the computational complexity point of view.}
Here $u,v,w$ are elements in the Weyl group $\mathcal{W}$ in type $Y$.
See \cite[$\S$4.2]{PR24o} for a brief overview, or \cite{AF} for a detailed exposition.
So far, we had focused on the type $A$ structure constants, i.e. the ${\sf GL}_n$ case.

Continuing this generalization, we may employ the type $Y$ Schubert polynomials $\Sc_w^Y$ of Billey--Haiman \cite{BH95} to compute these structure coefficients:
\[
\Sc_u^Y \ts \cdot \ts \Sc_v^Y \, = \, \sum_{w\in\mathcal{W}} \. c_{u,v}^w(Y) \. \Sc_w^Y\..
\]
Thus we consider the \defn{type $Y$ Schubert vanishing problem}:
$$
\SV(Y) \, := \, \big\{c^w_{u,v}(Y) =^? 0 \big\}.
$$

The main result of this paper extends to all classical type:

\smallskip

\begin{thm}[{\rm Schubert vanishing for all types}{}]\label{t:mainClassical}
$\SV(Y) \. \in \. \AM \cap \coAM$ \. assuming the \ts $\GRH$, for all \ts $Y\in\{A,B,C,D\}$.
\end{thm}

\begin{proof}
    By \cite[Theorem~2.4]{PR24o}, \ts  $\SV(Y)\in \coAM$ assuming the \ts $\GRH$ for $Y\in\{A,B,C\}$. Further, as noted in \cite[Remark~A.2]{PR24o}, the analogous result in type $D$ was prevented by a lingering determinantal equation $\det(\omega)=1$. The Determinant Lemma~\ref{l:det} resolves this issue, proving that \ts $\SV(D)\in \coAM$.\footnote{In the Appendix~$C$ of a revised version of \cite{PR24o}, written jointly with David Speyer, the authors circumvent this issue in a different way and also prove that \ts $\SV(D)\in \coAM$.}

    Further, the generality of \cite[Corollary~2.6]{Purbhoo06} gives a vanishing criterion in types $A, B, C, D$. For brevity, we suppress the details of the translation to types $B, C, D$ and instead review the general framework.

    Take the unipotent subgroup ${\sf N}\subset {\sf B}\subset {\sf G}$ of type $Y$. Let ${\mathfrak n}$ denote the Lie algebra of ${\sf N}$. Then for $w\in \mathcal{W}$, define $R_w:={\mathfrak n}\cap (w{\sf B}_{-}w^{-1})$, where ${\sf B}_{-}$ are the lower triangular matrices in ${\sf G}$. Then for generic elements $\rho,\omega,\tau\in {\sf N}$:
    \[ c_{u,v}^w\. = \. 0 \quad \Longleftrightarrow \quad
    \rho R_u\rho^{-1}+\omega R_v\omega^{-1}+\tau R_{w_\circ w}\tau^{-1} \, \subsetneq \, {\mathfrak n}.
    \]
    The argument used for Main Lemma~\ref{l:SV-HNP} works verbatim to translate the right-hand side into a system of polynomial equations. The only adjustment is that we may impose additional equations to ensure $\rho,\omega,\tau\in {\sf G}$, as specified in the relevant sections of \cite{PR24o}. Using the Determinant Lemma~\ref{l:det}, the resulting systems have polynomial size. This shows
    $\SV(Y)$ reduces to \HNP \  for $Y\in\{A,B,C,D\}$. Thus the result follows using Theorem~\ref{t:main-HNP}.
\end{proof}

\vskip.7cm

\subsection*{Acknowledgements}
We are grateful to Kevin Purbhoo for sharing his insights which partially
motivated this paper.  We thank Sara Billey, Nickolas Hein, Minyoung Jeon,
Allen Knutson, Leonardo Mihalcea, Greta Panova, Oliver Pechenik,
Maurice Rojas, Mahsa Shirmohammadi, Alex Smith, Frank Sottile, David Speyer,
Avi Wigderson, James Worrell, Weihong Xu, Alex Yong and Paul Zinn-Justin
for interesting discussions and helpful comments.
The first author was partially supported by the NSF grant CCF-2302173.
The second author was partially supported by the NSF MSPRF grant DMS-2302279.

\newpage
\appendix

\section{Schubert polynomials}\label{App:Schub}
Although we do not use Schubert polynomials to state or prove the main result,
the concept they represent is fundamental to fully understand the meaning
of Main Theorem~\ref{t:main}.   We thus include several different
definitions for reader's convenience.  We refer to \cite{Man01,Knu16}
for more on these definitions and further background.


\subsection{Divided differences} \label{App:Schub-div}
The following is the original definition due to Lascoux and Sch\"utzen-berger \cite{LS82}.
For a permutation \ts $w_\circ = (n,n-1,\ldots,2,1)$, let
$$
\Sch_{w_\circ} \. : = \. x_1^{n-1} x_2^{n-2} \. \cdots  \. x_{n-1}\..
$$
A permutation $w\in S_n$ is said to have a \emph{descent} at~$i$, if \ts $w(i)> w(i+1)$.
Denote by \ts $\Des(w)$ \ts the \emph{set of descents} \ts of~$w$, and by \ts $\des(\si):=|\Des(\si)|$
\ts the \emph{number of descents}. Define the \emph{divided difference operator}
$$
\pt_i F \, := \, \frac{F - s_i F }{x_i -x_{i+1}}\,,
$$
where the transposition $s_i:=(i,i+1)$ acts on $F\in \cc[x_1,\ldots,x_n]$ by transposing the variables.
For all $i \in \Des(w)$, let
$$
\Sch_{ws_i} \. := \. \pt_i \ts \Sch_w\.,
$$
and define all Schubert polynomials recursively.  It follows
that \ts $\Sch_w\in \zz[\bx]$ \ts are homogeneous polynomials of degree \ts $\inv(w)$.
Here \. $\inv(w):=\{(i,j) \, : \,i<j,\. w(i)>w(j)\}$ \. is the number of inversions in~$w$.

\smallskip

To show that Schubert polynomials are well defined, one needs to
check that
$$\pt_i\ts\pt_{j} \. = \. \pt_{j}\ts\pt_i \quad \text{for all \. $|i-j|\ge 2$, \ and}  \quad
\pt_i\ts\pt_{i+1}\ts \pt_i \. =  \. \pt_{i+1} \ts\pt_i\ts\pt_{i+1}\.,
$$
which follow from a straightforward computation.

This definition is elementary, easy to use, and can be generalized in various
directions.  The disadvantage of this definition is a nonobvious combinatorial
nature of the coefficients.  One can only conclude that \ts $[\bx^\al] \ts \Sch_w \in \zz$,
but not that \ts $[\bx^\al] \ts \Sch_w \in \nn$.

\smallskip

\subsection{Working forward, not backward} \label{App:Schub-for}
Using a standard embedding $S_n$ into $S_{n+1}$ by adding a fixed point $(n+1)$,
one can define a limit object $S_\infty$ to be the set of bijections $\si: \nn\to \nn$,
where $\si(i)=i$ for all but finitely many~$i$.  We can now define Schubert polynomials
forward, starting with $\Sch(1):=1$.  Use the
following rules:
$$
\Sch_{ws_i} \. := \. \pt_i \ts \Sch_w \quad \text{if \, $i \in \Des(w)$} \quad
\text{and} \quad \Sch_{ws_i} \. := \. 0 \quad \text{if \, $i \notin \Des(w)$}.
$$
In this setting, it is easy to see that Schubert polynomials are uniquely defined.
The existence becomes a substantive result, but other things become more apparent,
e.g.\ the symmetry properties of the construction.

Note that when \ts $\pt_i\ts F=0$, the polynomial $F$ is symmetric in variables $x_i$
and~$x_{i+1}$. Thus one can think of Schubert polynomials as partially symmetric.
In particular, it is easy to see that for the \emph{Grassmannian
permutations}, defined as permutations $w$ with $\des(w)=1$, we have $\Sch_w$ are
symmetric polynomials.  For example, when \ts
$u_{n,k}=(1,\ldots,n-k,n,n-k+1,\ldots,n-1)$, we have:
$$
\Sch_{u_{n,k}}(x_1,\ldots,x_n) \, =\, e_k(x_1,\ldots,x_n)  \, = \,
\sum_{1\le i_1< \ldots <i_k\le n} \. x_{i_1}  \cdots \. x_{i_k}
$$
is the \emph{elementary symmetric polynomial}.

\smallskip

\subsection{Pipe dreams}\label{App:Schub-pipe}
For a permutation \ts $w\in S_n$\ts, denote by $\RC(w)$  the set
of \ts \emph{RC-graphs} (also called \emph{pipe dreams}), defined as tilings
of a staircase shape with \emph{crosses} and \emph{elbows} as in the figure below,
such that:

{\small $(i)$} \, curves start in row $k$ on the left and end in column $w(k)$ on top, for all \ts $1\le k \le n$, and

{\small $(ii)$} \. no two curves intersect twice.

\nin
It follows from these conditions that every \ts $H \in \RC(w)$ \ts
has exactly \ts $\inv(w)$ \ts crosses.

\begin{figure}[hbt]
\begin{center}
	\includegraphics[height=2.5cm]{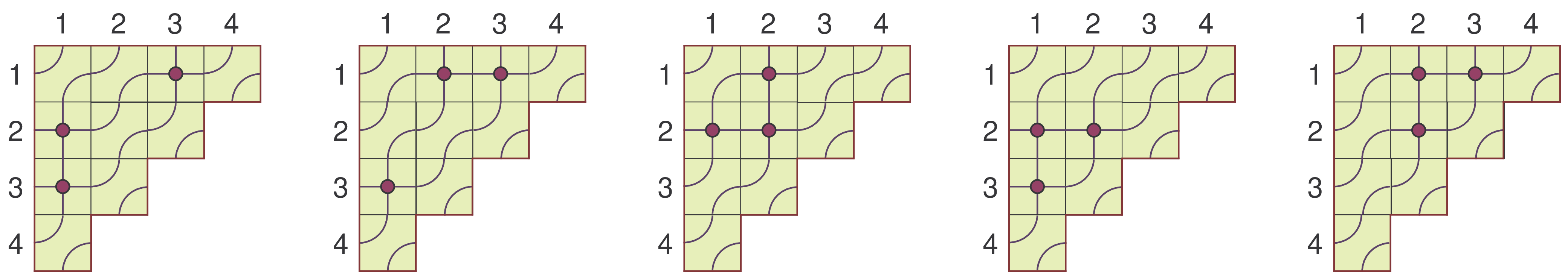}
\hskip-4.1cm
\caption{Graphs in \ts $\RC(1432)$ \ts and the corresponding Schubert
polynomial \.
 $\Sch_{1432} = x_1x_2x_3+x_1^2x_3+ x_1x_2^2+x_2^2x_3+x_1^2x_2$ \.
with monomials in this order.}
\label{f:RC}
\end{center}
\end{figure}


%
The {Schubert polynomial} \ts $\Sch_w\in \nn[x_1,x_2,\ldots]$ \ts
is defined as
\begin{equation}\label{eq:Schubert-def}
\Sch_w(\bx) \, := \, \sum_{H\ts\in\ts \RC(w)} \. \bx^H \quad \text{where} \quad
\bx^H \, := \, \prod_{(i,j) \. : \. H(i,j) \ts = \ts \boxplus} \. x_i \..
\end{equation}
In other words, \ts $\bx^H$ \ts is the product of \ts $x_i$'s
over all crosses \ts $(i,j)\in H$, see Figure~\ref{f:RC}.
As mentioned above, note that Schubert polynomials stabilize when fixed
points are added at the end, e.g.\ \ts $\Sch_{1432} = \Sch_{14325}$.

In this setting the combinatorial nature of Schubert polynomials is more transparent.
Notably, one can show that for Grassmannian permutations $w\in S_\infty$, Schubert
polynomial $\Sch_w$ coincides with the Schur function corresponding to
the partition given by the \emph{Rothe diagram}
$$\rR(w) \, := \, \big\{\big(w(j),i\big) \. : \. i<j, \. w(i)>w(j)\big\} \. \ssu \ts \nn^2.
$$
It follows from this definition that the coefficients
\ts $[\bx^\al] \ts \Sch_w$ \ts
are nonnegative integers, and moreover that they are in $\SP$ as a counting function.

\smallskip

\subsection{Geometric definition} \label{App:Schub-geom}

Let \ts ${\sf G}={\sf GL}_n(\cc)$ \ts  be the {general linear group}.
Take \ts ${\sf B}\subset {\sf G}$
to be the {Borel subgroup} of upper triangular matrices in ${\sf G}$. Similarly define \ts ${\sf B}_{-}\subset {\sf G}$ \ts the {opposite Borel subgroup} of lower triangular matrices in ${\sf G}$.

The \emph{complete flag variety} \ts is defined as \ts  $\Fl_n := {\sf G}/{\sf B}$.
Under the left action of ${\sf B}_{-}\ts$, the variety $\Fl_n$   has finitely many orbits \ts $X_w^{\circ}\ts$,
indexed by permutations $w\in S_n$.
These are called \emph{Schubert cells}.

 The \emph{Schubert varieties} $X_w$ are the Zariski closures of the Schubert cells~$X_w^{\circ}\ts$.
The \emph{Schubert classes} $\{\sigma_w\}_{w\in \mathcal{W}}$ are the
Poincar\'e duals of Schubert varieties.  These form a ${\mathbb Z}$--linear basis
of the cohomology ring \ts $H^{*}(\Fl_n)$.
\emph{Borel's ring isomorphism} \cite{Bor53} maps
\begin{align*}
    \Phi \. : \, H^{*}(\Fl_n) \. \longrightarrow \. \faktor{ \zz[x_1,x_2,\ldots,x_n]}{\langle e_i(x_1,x_2,\ldots,x_n) \, : \, i\in[n]\rangle}\,,
\end{align*}
where $e_i$ are elementary symmetric polynomials.
Schubert polynomials are polynomial representatives of Schubert classes:  $\Phi(\sigma_w)= \Sc_w\ts$.

In this setting, the {Schubert coefficients} $c_{u,v}^w$ are defined as structure constants:
 $$
    \sigma_u \smallsmile \sigma_v \, = \, \sum_{w\in {S_n}} \. c_{u,v}^{w} \. \sigma_{w}\..
$$

By the \emph{Kleiman transversality} \cite{Kleiman}, coefficients \ts $c_{u,v}^w$ \ts
count the number of points in the intersection of generically translated Schubert varieties:
\begin{equation}\label{eq:SchubStructureInt}
    c_{u,v}^w \,  = \,\#\big\{X_u(F_{\bullet}) \cap X_v(G_{\bullet}) \cap X_{w_\circ w}(E_{\bullet})\big\},
\end{equation}
where \ts $F_{\bullet}\ts$, \. $G_{\bullet}$ \ts and \ts $E_{\bullet}$ are generic flags.
This definition  implies the $S_3$-symmetries of Schubert coefficients:
$$c_{u,v}^{w_\circ w}  \,  = \, c_{v,u}^{w_\circ w} \, = \, c_{u,w}^{w_\circ v }
 \,  = \, c_{w,u}^{w_\circ v} \, = \, c_{v,w}^{w_\circ u }  \,  = \, c_{w,v}^{w_\circ u}
\,.
$$

\bigskip


\section{Quotes and historical remarks}\label{s:finrem}

\subsection{Formulation of the problem}\label{ss:finrem-hist}
The problem of finding a combinatorial interpretation of Schubert
coefficients goes back to Lascoux and Sch\"utzenberger in 1980s,
and was restated by numerous authors.  As the area evolved, so
did the language and the formulation of the problem.
For example, in his celebrated survey, Stanley states the problem
as follows:

\smallskip

\begin{center}\begin{minipage}{13.8cm}%
{\emph{``Find a combinatorial interpretation of the
`Schubert intersection coefficients' \ts $c^w_{u,v}\ts$,
thereby combinatorially reproving that they are
nonnegative.''}~\cite[Problem~11]{Sta00}}
\end{minipage}\end{center}

\smallskip

\nin
In the introduction to his monograph, Manivel singles out the problem as the main mystery
in the area:

\medskip

\begin{center}\begin{minipage}{13.8cm}%
{\emph{``We note that Schubert polynomials are far from having revealed all of their
secrets. We know almost nothing, for example, about their multiplication, and
about a rule of Littlewood--Richardson type which must govern them.''}~\cite[p.~3]{Man01}}
\end{minipage}\end{center}

\smallskip

\nin
Lenart motivates the problem by the geometry, and as an effort to avoid the geometry
altogether (cf.~\cite{Ass23}).
He also singles out the vanishing problem as a motivation:

\medskip

\begin{center}\begin{minipage}{13.8cm}%
{\emph{``A famous open problem in algebraic combinatorics, known as the Schubert problem
$[...]$ is to find a combinatorial description of the Schubert structure
constants (and, in particular, a proof of their nonnegativity which bypasses geometry). The importance
of this problem stems from the geometric significance of the Schubert structure constants, and from
the fact that a combinatorial interpretation for these coefficients would facilitate a deeper study of
their properties (such as their symmetries, vanishing, etc.). The Schubert problem proved to be a
very hard problem, resisting many attempts to be solved.''}~\cite{Len10}}
\end{minipage}\end{center}

\smallskip

\nin
Despite many remarkable developments, these sentiments continue to hold as
underscored by Knutson, who used a starkly different language:

\smallskip

\begin{center}\begin{minipage}{13.8cm}%
{\emph{``We cannot emphasize strongly enough that the name of the game is to give
{\bf {manifestly nonnegative formul{\ae}}} for the
$[$Schubert coefficients$]$.''}~\cite[$\S$1.4]{Knu22}}\footnote{Original emphasis.}
\end{minipage}\end{center}

\smallskip

\nin
Knutson then emphasizes the vanishing problem as the first motivation:\footnote{Two
other motivations are computational efficiency and ``possibility for
categorification''. } 

\medskip

\begin{center}\begin{minipage}{13.8cm}%
{\emph{``For applications (including real-world engineering applications) it is more
important to know that some structure constant $c$ is positive, than it is to know
its actual value. This is much more easily studied with a noncancelative
formula.''}~(ibid.)}
\end{minipage}\end{center}

\smallskip
\subsection{Substance of the problem}\label{ss:finrem-nonSP}
There is a great deal of uncertainty in algebraic combinatorics as to what exactly
constitutes a ``combinatorial interpretation''.  This is best illustrated by
the following formulation of the [main problem] in the most recent monograph:

\smallskip

\begin{center}\begin{minipage}{13.8cm}%
{\emph{``Find an interpretation for the Schubert structure constants in terms
of counting some sort of combinatorial objects such as paths in Bruhat order, Mondrian
tableaux, labeled diagrams, permutation arrays, or $n$-dimensional chess
games.''}~\cite[Open Problem~3.112]{BGP25}}
\end{minipage}\end{center}

\smallskip

\nin
These combinatorial objects are all in $\SP$, making the problem harder and more narrow
than it already is (or simpler, since without any complexity assumptions \emph{any number}
is the number of \emph{some} paths in Bruhat order).
We maintain that $\SP$ as the only known robust notion of a ``combinatorial interpretation'',
and refer the reader to \cite{IP22,Pak-OPAC} for the explanation behind this reasoning.
We only mention in passing that combinatorial objects in the problem above come from
well-known attempts to resolve the problem.  Curiously, the authors hedge
themselves:

\medskip

\begin{center}\begin{minipage}{13.8cm}%
{\emph{``Note, the Schubert structure constants already count the number of points
in a certain type of generic $0$-dimensional intersection $[...]$
Perhaps one could call this a combinatorial interpretation, since they do count
something! However, it is very difficult to test if flags are truly in generic position, even
though presumably almost anything you could choose would suffice.''}~\cite[Remark~3.113]{BGP25}}
\end{minipage}\end{center}

\medskip

\nin
This remark goes straight to the core of the issue and underscores the need for
the formal approach.  Fundamentally, this paper is an attempt to understand
the computational hardness of counting these intersections. 

\smallskip
\subsection{Complexity of the problem}\label{ss:finrem-complexity}
Prior to \cite{PR24o,PR24b}, the effort to analyze the hardness of
the Schubert vanishing problem was largely unsuccessful:

\smallskip

\begin{center}\begin{minipage}{13.8cm}%
{\emph{``It is well known that solving Schubert problems are `hard'.
To our knowledge, no complete analysis of the algorithmic complexity is
known.''}~\cite[p.~41]{BV08}}
\end{minipage}\end{center}

\smallskip

\nin
In this quote, Billey and Vakil are fully cognizant that counting
$0$-dimensional intersections can be the basis of the algorithm, as
they describe in the paper.  After employing a mixture of theoretical
analysis and experimental evidence, they conclude:

\smallskip

\begin{center}\begin{minipage}{13.8cm}%
{\emph{``Of course this allows one in theory to solve all Schubert problems,
but the number and complexity of the equations conditions grows quickly to
make this prohibitive for large~$n$.''}~\cite[p.~24]{BV08}}
\end{minipage}\end{center}

\smallskip

\nin
Nothing in this paper suggests that computing Schubert coefficients
can be made efficient; we are nowhere close to practical applications.
Note, however, our lifted formulation approach is different from that by
Billey--Vakil's experimental effort:

\smallskip

\begin{center}\begin{minipage}{13.8cm}%
{\emph{``It is well known in that solving more equations with fewer variables
is not necessarily an improvement. More experiments are required to characterize
the `best' method of computing Schubert problems. We are limited
in experimenting with this solution technique to what a symbolic
programming language like Maple can do in a reasonable period of time.
The examples in the next section will illustrate how this technique is useful
in keeping both the number of variables and the complexity of the rank
equations to a minimum.'}~\cite[p.~43]{BV08}}
\end{minipage}\end{center}

\smallskip

\nin
In contrast, we are happy to increase the number of variables to ensure
that resulting polynomials have a poly-size support, at which point
Theorem~\ref{t:main-HNP} can be applied.


\end{document}